\documentclass[11pt]{amsart}
\usepackage{graphicx}
\usepackage{graphics}

\usepackage{amsmath}
\usepackage{amscd}
\usepackage{latexsym}
\usepackage[all]{xy}

\begin{document}
\textwidth 5.5in
\textheight 8.3in
\evensidemargin .75in
\oddsidemargin.75in

\newtheorem{lem}{Lemma}[section]
\newtheorem{conj}[lem]{Conjecture}
\newtheorem{defn}[lem]{Definition}
\newtheorem{thm}[lem]{Theorem}
\newtheorem{cor}[lem]{Corollary}
\newtheorem{prob}[lem]{Problem}
\newtheorem{exm}[lem]{Example}
\newtheorem{rmk}[lem]{Remark}
\newtheorem{que}[lem]{Question}
\newtheorem{prop}[lem]{Proposition}
\newtheorem{clm}[lem]{Claim}
\newcommand{\p}[3]{\Phi_{p,#1}^{#2}(#3)}
\def\Z{\mathbb Z}
\def\R{\mathbb R}
\def\g{\overline{g}}
\def\odots{\reflectbox{\text{$\ddots$}}}
\newcommand{\tg}{\overline{g}}
\def\proof{{\bf Proof. }}
\def\ee{\epsilon_1'}
\def\ef{\epsilon_2'}
\title{Finite order corks}
\author{Motoo Tange}
\thanks{The author was partially supported by JSPS KAKENHI Grant Number 26800031}
\address{Institute of Mathematics, University of Tsukuba,
 1-1-1 Tennodai, Tsukuba, Ibaraki 305-8571, Japan}
\email{tange@math.tsukuba.ac.jp}
\date{\today}
\maketitle
\begin{abstract}
We show that for any positive integer $m$, there exist order $n$ Stein corks $(C_{n,m},\tau_{n,m}^C)$.
The boundaries are cyclic branched covers of slice knots embedded in the boundary of a cork.
By applying these corks to generalized forms, we give a method producing examples of many finite order corks, which are possibly not Stein cork.
The examples of the Stein corks give $n$ homotopic and contactomorphic but non-isotopic Stein filling contact structures for any $n$.
\end{abstract}
%
%
\section{Introduction}
\label{intro}
\subsection{4-manifold and cork.}
If two smooth manifolds $X_1$ and $X_2$ are homeomorphic but non-diffeomorphic, then
we say that $X_1$ and $X_2$ are {\it exotic}.
Let $X$ be a smooth manifold and $Y$ a codimension $0$ submanifold in $X$.
We denote the cut-and-paste $(X-Y)\cup_\phi Z$ by $X(Y,\phi,Z)$.
In the case of $Y=Z$, we denote such a surgery by $X(Y,\phi)$ and say it {\it a twist}.
For simply-connected closed exotic 4-manifolds the following theorem is well-known.
\begin{thm}[\cite{CFHS},\cite{Mat},\cite{AM}]
For any simply-connected closed exotic 4-manifolds $X_1,X_2$, there exist a contractible Stein manifold
${\mathcal C}$, an embedding $\mathcal{C}\hookrightarrow X_1$, and a self-diffeomorphism $t:\partial{\mathcal C}\to \partial{\mathcal C}$ with $t^2=\text{id}$
such that
$$X_1({\mathcal C},t)=X_2.$$
\end{thm}
This theorem says that a pair of a contractible Stein manifold and a self-diffeomorphism of the boundary causes ``exoticity'' of simply-connected closed smooth 4-manifolds.
Studying corks is important for understanding the exotic phenomenon of 4-manifolds.
We give a definition of cork in a generalized form.
\begin{defn}[Cork]
Let ${\mathcal C}$ be a contractible 4-manifold and $t$ a self-diffeomorphism $\partial \mathcal{C}\to \partial \mathcal{C}$ on the boundary.
If $t$ cannot extend to a map $\mathcal{C}\to\mathcal{C}$ as a diffeomorphism, then $(\mathcal{C},t)$ is called a cork.

For a pair of exotic two 4-manifolds $X_1,X_2$ and a smooth embedding $\mathcal{C}\hookrightarrow X_1$,
if $X_1(\mathcal{C},t)=X_2$, then $(\mathcal{C},t)$ is called a cork for $X_1,X_2$.
We call the deformation $X_1\to X_1(\mathcal{C},t)$ {\it a cork twist}.
\end{defn}
Note that this definition of cork are weaker than that of the usual one in terms of the following two points.
The definition here does not assume that $\mathcal{C}$ is Stein and $t$ satisfies $t^2=\text{id}$.
If the 4-manifold $\mathcal{C}$ of a cork $(\mathcal{C},t)$ is Stein, then $(\mathcal{C},t)$ is called a {\it Stein cork}.
\subsection{Order $n$ cork}
We define order $n$ cork.
\begin{defn}[Order $n$ cork]
\label{orderndefinition}
Let $({\mathcal C},t)$ be a cork.
If the following conditions are satisfied, then we call $(\mathcal{C},t)$ an order $n$ cork:
\begin{enumerate}
\item the composition $t\circ t\circ \cdots \circ t=t^i$ $(0<i< n)$ cannot extend to any diffeomorphism $\mathcal{C}\to \mathcal{C}$.
\item $t^n$ can extend to a diffeomorphism $\mathcal{C}\to \mathcal{C}$. 
\end{enumerate}
This number $n$ is called the order of the cork $({\mathcal C},t)$.

Let $\mathbb{X}$ be a pair of $n$ mutually exotic 4-manifolds $\{X=X_0,\cdots,X_{n-1}\}$.
If there exists an embedding $\mathcal{C}\hookrightarrow X$ such that $X_i=X(\mathcal{C},t^i)$, then $(\mathcal{C},t)$ is called a cork for this collection $\mathbb{X}$.
\end{defn}
The existence of finite order corks has been not known except for order 2.
\subsection{Aims.}
Let $\mathcal{C}$ be a contractible 4-manifold.
One of our aims of this article is to construct infinite families of order $n$ corks for each $n>1$.
Another aim is to give a technique to show that the map $t$ for a twist $(\mathcal{C},t)$ cannot extend to inside $\mathcal{C}$ as a diffeomorphism, i.e, $(\mathcal{C},t)$ is a cork.

Freedman's result \cite{F} says that the diffeomorphism on $\partial \mathcal{C}$ extend to a {\it self-homeomorphism} on $\mathcal{C}$.

Infinite order corks are not known so far.
\subsection{Results.}
\label{results}
Let $(C(m),\tau(m))$ be a pair defined as the handle diagram as in {\sc Figure}~\ref{Ack} and the diffeomorphism $\tau(m)$ is order $2$.
\begin{figure}[htbp]
\begin{center}\includegraphics[width=.5\textwidth]{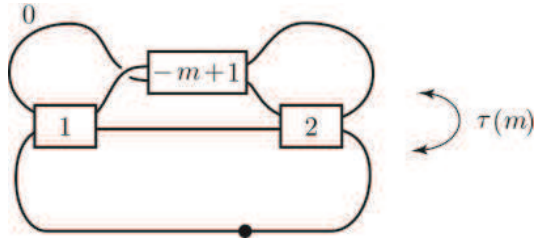}
\caption{Corks $(C(m),\tau(m))$.}
\label{Ack}
\end{center}
\end{figure}
In the case of $m=1$, $(C(1),\tau(1))$ is the Akbulut cork in \cite{A1}.
$(C(m),\tau(m))$ is an example of an order 2 Stein cork (Proposition~\ref{corkorder2}).
By positioning several copies of the attaching spheres of this cork on the boundary of the 0-handle, we give examples of finite corks.
One of the main theorems is the following.
\begin{thm}
\label{main}
Let $n,m$ be integers with $n>1$ and $m>0$.
There exists an order $n$ cork $(C_{n,m},\tau_{n,m}^C)$.
The handle decomposition and the map $\tau_{n,m}^C$ are described in {\sc Figure}~\ref{beisotopy}.
Furthermore, $C_{n,m}$ is an order $n$ Stein cork.
\end{thm}
\begin{figure}[htbp]
\begin{center}
\includegraphics[width=.5\textwidth]{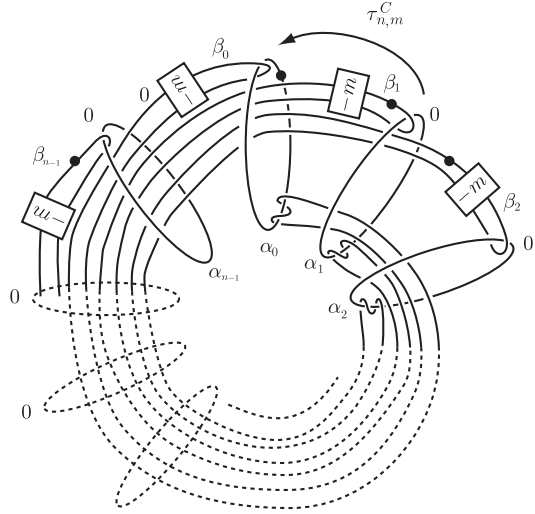}
\caption{The handle decomposition of $C_{n,m}$.}
\label{beisotopy}
\end{center}
\end{figure}
\begin{figure}[htbp]
\begin{center}
\includegraphics[width=.4\textwidth]{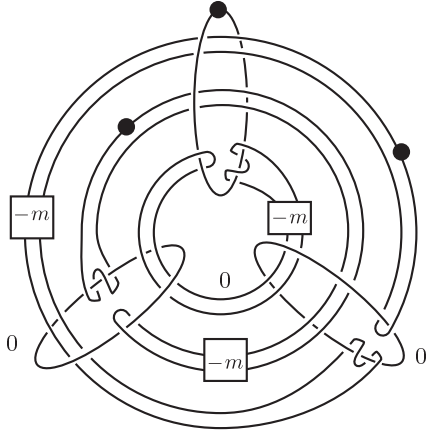}
\caption{The handle decomposition of $C_{3,m}$ after an isotopy.}
\label{defofCn}
\end{center}
\end{figure}
The number $-m$ in any box stands for a $-m$ full twist.

We will define other variations $D_{n,m}$ and $E_{n,m}$ in Section~\ref{OvDE}.
Here we give rough definitions of them.
$D_{n,m}$ is obtained by the exchange of all dots and 0s of $C_{n,m}$.
$E_{n,m}$ is a 4-manifold modified as in {\sc Figure}~\ref{henka} of $C_{n,m}$.
The case of $n=3$ is described as in {\sc Figure}~\ref{elliptic}.
The diffeomorphisms $\tau^D_{n,m}$ and $\tau_{n,m}^E$ are the rotations by angle $2\pi/n$ in the same way as $\tau_{n,m}^C$.

The reason why we treat these examples is to show the existence of many corks {\it without the direct aid of Stein structure}.
In other words, even when we do not know whether $D_{n,m}$ and $E_{n,m}$ are Stein manifolds,
our technique can show that $\tau_{n,m}^D$ and $\tau_{n,m}^E$ cannot extend to inside $D_{n,m}$ and $E_{n,m}$ as diffeomorphisms respectively.
\begin{figure}[htpb]
\begin{center}
\includegraphics[width=.5\textwidth]{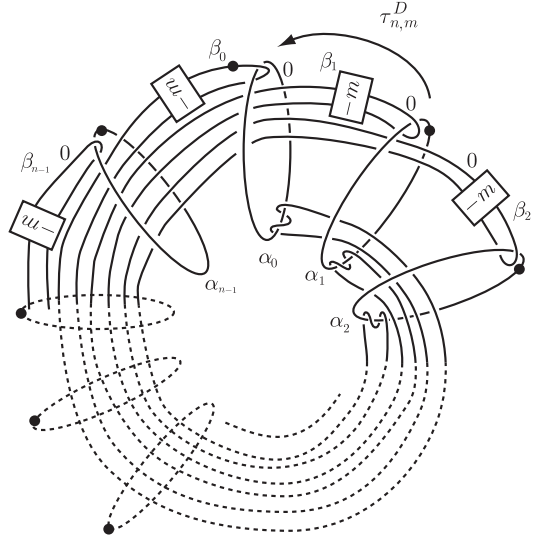}
\caption{The handle decomposition of $D_{n,m}$.}
\label{handleD}
\end{center}
\end{figure}
\begin{figure}[htbp]
\begin{center}
\includegraphics{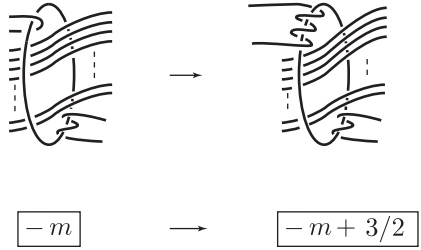}
\caption{A modification of $C_{n,m}$ into $E_{n,m}$.}
\label{henka}
\end{center}
\end{figure}
\begin{figure}[htpb]
\begin{center}\includegraphics[width=.5\textwidth]{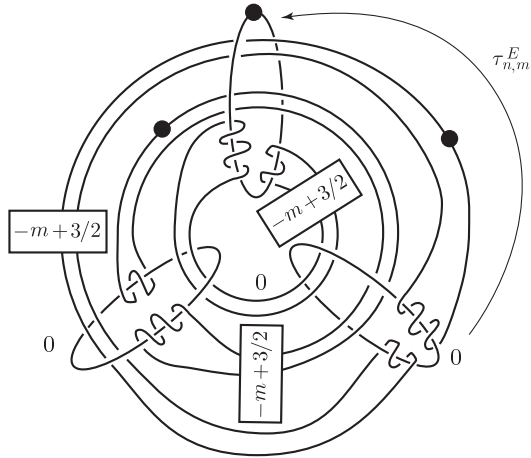}
\caption{Handle decomposition of $E_{3,m}$ and a diffeomorphism $\tau_{n,m}^E$.}
\label{elliptic}
\end{center}
\end{figure}
Here, as examples, we state the second and third main theorems in the form including technical statements.
\begin{thm}[Cork-ness of $(D_{n,m},\tau_{n,m}^D)$]
\label{contra}
Let $n$ be an integer with $n>2$.
Then there exists an embedding $D_{n,m}\subset D_{2,m}$ such that for the embedding, there exists a diffeomorphism $\psi:D_{2,m}\to D_{2,m}(D_{n,m},\tau_{n,m}^D)$ such that the diffeomorphism induces $\tau_{2,m}^D$ on the boundaries.

In particular, $(D_{n,m},\tau_{n,m}^D)$ is an order $n$ cork.
\end{thm}
Note that the first statement does not mean the induced diffeomorphism $\partial D_{2,m}\to \partial D_{2,m}$ can extend to inside a diffeomorphism.
The induced map is the restriction of a rotation of $D_{2,m}-D_{n,m}$ of $\psi$ to the one component of the boundary.
\begin{thm}[Cork-ness of $(E_{n,m},\tau_{n,m}^E)$]
\label{thmE}
For a positive integers $n,m$ there exists a sufficient large integer $l$ and an embedding $E_{n,m}\hookrightarrow V_{l,n}:=E(l)\#n\overline{{\mathbb C}P^{2}}$ such that for any $0<i<n$ the twist $V_{l,n}(E_{n,m},(\tau_{n,m}^E)^i)$ is diffeomorphic to $(2l-1){\mathbb C}P^2\#(10l+n-1)\overline{{\mathbb C}P^2}$.

In particular, $E_{n,m}$ is an order $n$ cork.
\end{thm}
The integer $l$ in this theorem is independent of $m$.
The point is that all the nontrivial twists $(E_{n,m}, (\tau_{n,m}^E)^i)$ for an embedding $E_{n,m}\hookrightarrow V_{l,n}$ produces exotic structures.
Whether there exists a finite order cork for the collection of mutually exotic 4-manifolds is not known yet.
\subsection{More exotic 4-manifolds}
\label{more}
In Section~\ref{exo}, we give a 4-manifold $W_{n,m}$ and embedding $C_{n,m}\hookrightarrow W_{n,m}$.
The cork twist is a candidate of the collection of mutually exotic 4-manifolds.
\begin{prop}
\label{collexo}
Let $i$ be an integer with $0< i\le n-1$.
There exists a simply-connected non-spin 4-manifold $W_{n,m,i}$ with $b_2=b^-=n(n-1)/2$ and a homology sphere boundary.
The manifold $W_{n,m,i}$ is obtained by an order $n$ cork twist of a Stein manifold $W_{n,m}$.
Each $W_{n,m,i}$ is the $i$ times blow-ups of a 4-manifold $W_{n,m,i}'$ and is exotic to $W_{n,m}$.
\end{prop}
We do not know whether $W_{n,m,i}'$ is a minimal 4-manifold.
\begin{rmk}
In the definition of order $n$ cork in \cite{TM1}, we imposed the condition that the order of $t$ is $n$.
Here we slightly change it to the weaker condition (2) in Definition~\ref{orderndefinition}.
\end{rmk}
This remark says that the order of $t$ as a map is ``not'' necessary to be the same as the order of cork.
In the last section we will illustrate an example having the difference.
\begin{prop}
\label{exm1}
Let $(F,\kappa)$ be a pair of a 4-manifold and diffeomorphism as in {\sc Figure}~\ref{orderdiff}.
The map $\kappa$ on $\partial F$ has order 2 as a cork, however it is an order $4$ as a diffeomorphism.
\end{prop}

\subsection{An action on Heegaard Floer homology.}
S. Akbulut and C. Karakurt in \cite{AK} showed that the twist map $t$ of an order $2$ Stein cork $(\mathcal{C},t)$ 
consisting of a symmetric diagram consisting of a dotted 1-handle and 
a 0-framed 2-handle induces an action on the Heegaard Floer homology $HF^+(\partial \mathcal{C})$ non-trivially as an involution.
We show that an order $n$ Stein cork $(C_{n,m},\tau_{n,m}^C)$ also induces an order $n$ map on Heegaard Floer homology $HF^+(\partial C_{n,m})$.
\begin{thm}
\label{action}
Let $n,m$ be positive integers with $n>1$.
The maps $\{(\tau_{n,m}^C)^i|i=0,\cdots,n-1\}$, which is isomorphic to ${\mathbb Z}/n{\mathbb Z}$, act effectively on the Heegaard Floer homology $HF^+(\partial C_{n,m})$.
\end{thm}
An action of a group $G$ on a set $S$ is called {\it effective}, if for any $e\neq g\in G$ there exists an element $x\in S$ such that $g\cdot x\neq x$.
From this theorem, we immediately the following proposition.
\begin{prop}
\label{contact}
There exist $n$ Stein filling contact structures on $\partial C_{n,m}$ such that they are homotopic and contactomorphic but non-isotopic each other.
\end{prop}
\section*{Acknowledgements}
This topic was come upon by the author while speaking with Masatsuna Tsuchiya on November in 2014.
I thank him for arguing with me.
I gave a talk about finite order corks in the Handle friendship seminar in Tokyo Institute of Technology
and also thank the audience for giving the helpful comments.
I also thank Kouichi Yasui and Tetsuya Abe for giving me useful comments in the earlier version.
\section{A 4-manifold $X_{n,m}({\bf x})$ for a $\{\ast,0\}$-sequence.}
\subsection{Extendability of a composite boundary diffeomorphism.}
\label{ex}
We defined the order of cork in Section~\ref{intro}.
The definition is slightly different from that defined in \cite{TM1} as mentioned in Section~\ref{more}.
To note the well-definedness of the order of cork, we show the following fundamental lemma.
\begin{lem}
Let $X$ be a manifold and $\varphi,\psi$ boundary self-diffeomorphisms $\partial X\to \partial X$.
If $\varphi,\psi$ extend to inside $X\to X$ as a diffeomorphism, then so does $\varphi\circ \psi$.
\end{lem}
\begin{proof}
We attach the cylinder $\partial X\times I$ to $X$ by using $\varphi:\partial X\times\{1\}\to \partial X$.
The extendability problem of $(X,\varphi\circ \psi)$ is equivalent to that for a pair $(X\cup_\varphi (\partial X\times I),\psi)$, where $\psi$ induces a map $\psi:\partial X\times \{1\}\to \partial X\times \{1\}$.
Since $\varphi: \partial X\to \partial X$ extends to inside, by an identification $X\to \varphi(X)$,
$X\cup_\varphi (\partial X\times I)$ is diffeomorphic to $X$ with the boundary point-wise fixed.
The problem is reduced to the extendability problem of $\psi:\partial X\to \partial X$.
From the assumption, therefore, $\psi\circ \varphi$ extends to inside.
\qed\end{proof}
From this lemma if a self-diffeomorphism $\varphi$ on the boundary of a manifold extends to inside as a diffeomorphism, then so does $\varphi^n$.
\subsection{Constructions of $X_{n,m}({\bf x})$.}
\label{constructC}
Let $m$ be a positive integer and $n$ an integer with $n>0$.
Let $L_{n,m}$ denote $2n$ components link (located like a wheel) as in the diagram in {\sc Figure}~\ref{beisotopy}.
We denote the components by 
$$L_{n,m}=\{\alpha_0,\alpha_1,\cdots,\alpha_{n-1},\beta_0,\beta_1,\cdots,\beta_{n-1}\},$$
where the number $-m$ in any box stands for a $-m$ full twists.
The $n$-components $\{\alpha_i|i=0,1,\cdots, n-1\}$ (here called them radial components) of them lie in the radial direction about the center of the rotation.
The rest $n$-components $\{\beta_i|i=0,1,\cdots,n-1\}$ in $L_{n,m}$ (here we call them circular components) are put in the form that makes a circuit from a radial component to the same component.
The linking number between $\alpha_i$ and $\beta_i$ is one and other linking numbers between $\alpha_i$ and $\beta_j$ ($i\neq j$) are all zero.

Let ${\bf x}=(x_0,x_1,\cdots,x_{n-1})$ be a $\{\ast,0\}$-sequence with $x_i=\ast$ or $0$.
If $x_i=\ast$, then we describe a dot on the circle $\alpha_i$ and $0$ near the circle $\beta_i$.
If $x_i=0$, then we describe $0$ near the circle $\alpha_i$ and a dot on the circle $\beta_i$.
According to this process, we get a framed link diagram.
$X_{n,m}({\bf x})$ is defined to be a 4-manifold obtained by this framed link diagram.

From the construction, we have $H_\ast(X_{n,m}({\bf x}),\mathbb{Z})\cong H_\ast(B^4,\mathbb{Z})$, where $B^4$ is the 4-ball.
We note that $X_{1,m}(\ast)=X_{1,m}(0)=C(m)$ holds.

We can also construct $X_{n,m}({\bf x})$ in terms of branched cover of $C(m)$.
The $n$-fold branched cover of $C(m)$ along the slice disk of $K_{n,m}$ as in {\sc Figure}~\ref{slice} is $X_{n,m}(0,\cdots,0)$.
Doing several times cork twists $(C(m),\tau(m))$ for $\{\alpha_i,\beta_i\}$ with $x_i=\ast$ in the sequence ${\bf x}$, we get $X_{n,m}({\bf x})$.
\begin{figure}[htbp]
\begin{center}
\includegraphics{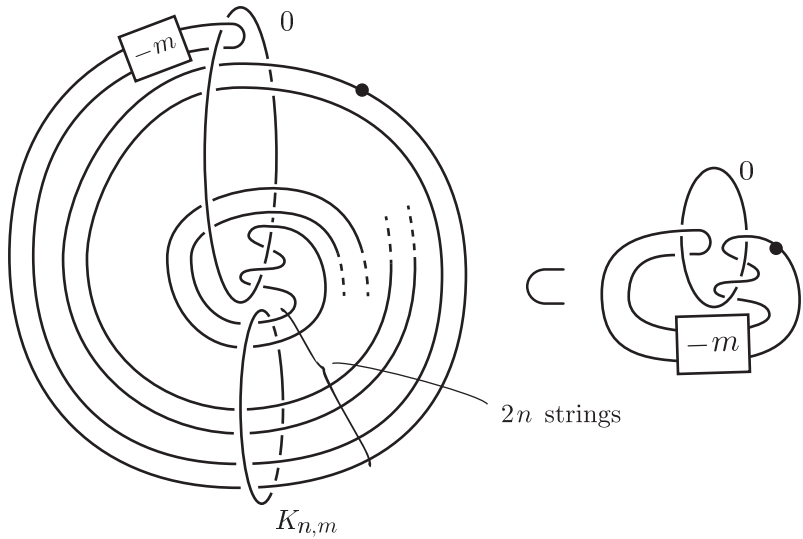}
\caption{A slice knot $K_{n,m}$ on $\partial C(m)$.}
\label{slice}
\end{center}
\end{figure}
\begin{lem}
\label{allcontra}
Let ${\bf x}$ be any $\{\ast,0\}$-sequence.
$X_{n,m}({\bf x})$ is a contractible 4-manifold.
\end{lem}
\begin{proof}
We show $X_{n,m}({\bf x})$ is simply-connected.
For $0\le i\le n-1$ exchanging $x_i$ as $\ast\to 0$ corresponds to a cork twist of $C(m)=0\text{-handle}\cup\{\alpha_i,\beta_i\}$.
Thus, this exchange does not change the fundamental group, i.e., $\pi_1(X_{n,m}(\cdots,\ast,\cdots))\cong \pi_1(X_{n,m}(\cdots,0,\cdots))$.

We may show $\pi_1(X_{n,m}(0,\cdots,0))$ is trivial.
Dotted circles in the diagram of $X_{n,m}(0,\cdots,0)$ have a separated position with each dotted circles.
Since any $\alpha_i$ does not link with $\beta_j$ with $i\ne j$.
The presentation of the fundamental group is the same as $\pi_1(\natural nC(m))$.
Thus, this group is the trivial group.
This means $\pi_1(X_{n,m}({\bf x}))=e$ for any $\{\ast,0\}$-sequence.

Thus $X_{n,m}({\bf x})$ is a contractible 4-manifold.
\qed\end{proof}
\begin{defn}[$C_{n,m}$, $D_{n,m}$, and $F_{n,m}$]
\label{CDF}
We define $C_{n,m},D_{n,m}$ and $F_{n,m}$ to be
$$C_{n,m}=X_{n,m}(\ast,0,\cdots,0),$$
$$D_{n,m}=X_{n,m}(0,\ast,\cdots,\ast),$$
and
$$F_{n,m}=X_{2n,m}(0,\ast,0\ast \cdots,0,\ast).$$
See {\sc Figure}~\ref{beisotopy} and~\ref{handleD} for $C_{n,m}$ and $D_{n,m}$
and see {\sc Figure}~\ref{orderdiff} for $F_{2,m}$.
\end{defn}
We show the following proposition.
\begin{prop}
\label{property}
\begin{enumerate}
\item $C_{2,m}=D_{2,m}$ holds.
\item $C_{n,m}, D_{n,m}, E_{n,m}$ and $F_{n,m}$ are contractible 4-manifolds.
\end{enumerate}
\end{prop}
\begin{proof}
(1) The handle decompositions of $C_{2,m}$ and $D_{2,m}$ are the same and the diffeomorphisms $\tau_{2,m}^C$ and $\tau_{2,m}^D$ are the exchange of dots and 0s.

(2) Lemma~\ref{allcontra} says that $C_{n,m}$, $D_{n,m}$ and $F_{n,m}$ are contractible.
We can also show that $E_{n,m}$ is contractible in the same way as Lemma~\ref{allcontra}.
\end{proof}
\subsection{A diffeomorphism $\tau_{n,m}^X$.}
Moving each component of $L_{n,m}$ as
$$\alpha_i\to \alpha_{i-1}$$
and
$$\beta_i\to \beta_{i-1},$$
we get a diffeomorphism $\tau_{n,m}^X:\partial X_{n,m}({\bf x})\to \partial X_{n,m}({\bf x})$.
Here we consider the suffices as elements in ${\mathbb Z}/n{\mathbb Z}$.
The diffeomorphism $\tau_{n,m}^X$ when $X=C,D$ is a rotation by angle $2\pi/n$ as defined in Section~\ref{results}.
\subsection{An isotopy}
We give an isotopy of $L_{n,m}$ as described in {\sc Figure}~\ref{isotopy2}.
First, we move $\beta_0$ to the innermost position.
Next, we move $\beta_1$ to the second innermost position.
In the same way as above we move all $\beta_i$.
By using the isotopy, we can get a handle diagram presentation of $X_{n,m}({\bf x})$.
As a diagram after isotopy, see {\sc Figure}~\ref{defofCn}.

\begin{figure}[htbp]
\begin{center}
\includegraphics[width=.9\textwidth]{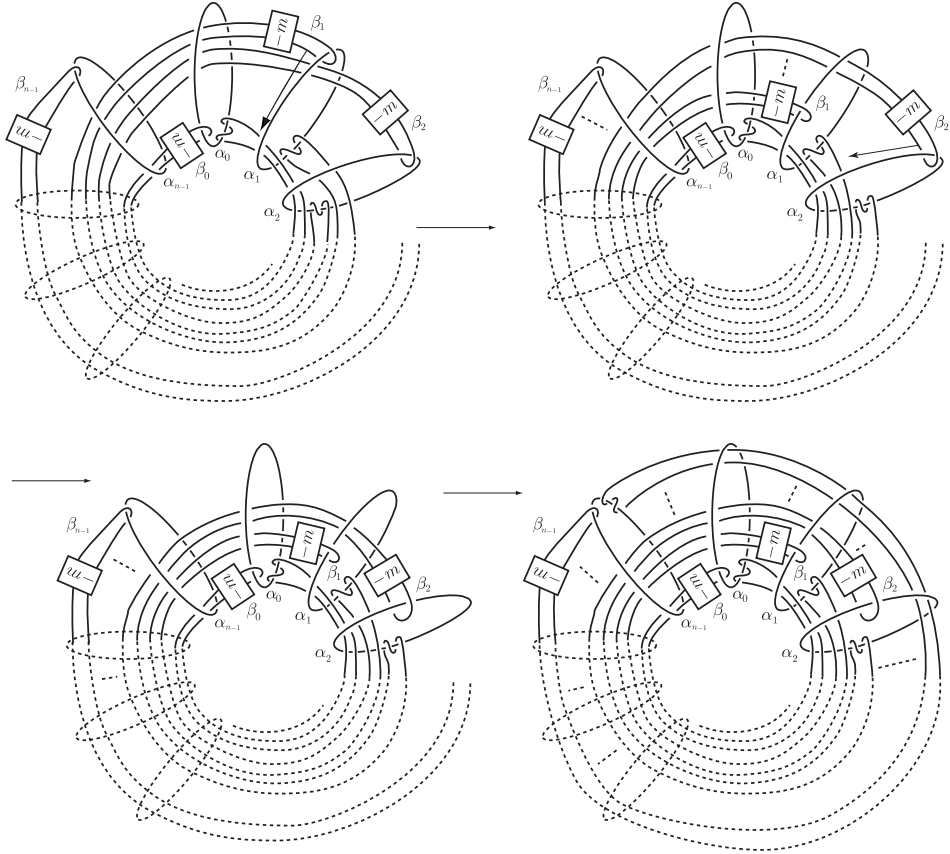}
\caption{An isotopy of a handle diagram of $C_{n,m}$.}
\label{isotopy2}
\end{center}
\end{figure}
%
%
\subsection{Another variation $E_{n,m}$.}
\label{OvDE}
\begin{defn}
We define $E_{n,m}$ to be the manifold obtained by the modification of $C_{n,m}$ as in {\sc Figure}~\ref{henka}.
\end{defn}
The handle diagram for $E_{3,m}$ after the same isotopy as above is {\sc Figure}~\ref{elliptic}.
In the same way as $C_{n,m}$ or $D_{n,m}$ each pair of a dotted 1-handle and 0-framed 2-handle with linking number 1 in $E_{n,m}$ consists of the cork $C(m)$ (see {\sc Figure}~\ref{elliptic2}).
\begin{figure}[htpb]
\begin{center}\includegraphics[width=.8\textwidth]{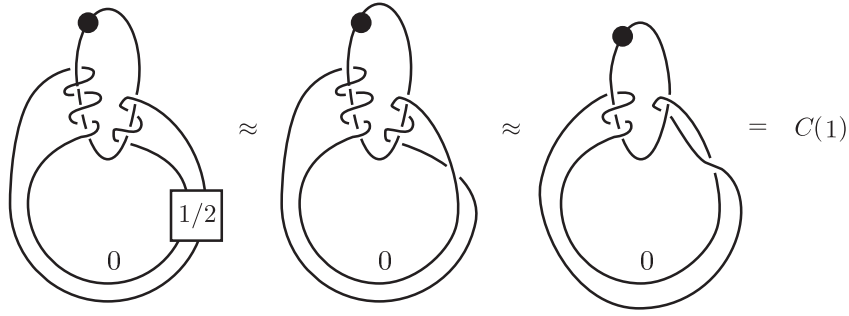}
\caption{The Akbulut cork $C(1)$ embedded in $E_{n,1}$.}
\label{elliptic2}
\end{center}
\end{figure}
\section{Proofs of main results}
\subsection{The cork-ness of $C_{n,m}$.}
Due to Gompf's result \cite{G} in order to see that $C(m)$ and $C_{n,m}$ admit Stein structure, we may deform the handle diagrams of $C(m)$ and $C_{n,m}$ into Legendrian links with some Thurston-Bennequin condition.
On the standard position of $\#nS^2\times S^1$ in \cite{G}, which is the boundary of the end sum $\natural nD^3\times S^1$, we put Legendrian links with all framings $tb-1$, where $tb$ is the Thurston-Bennequin number of each Legendrian knot.

Here we show the following.
\begin{prop}
\label{corkorder2}
$(C(m),\tau(m))$ is an order 2 Stein cork.
\end{prop}
\begin{proof}
The Stein structure on $C(m)$ is due to {\sc Figure}~\ref{Ackstein}.
Here, the box with number $-m+1$ stands for a ($-m+1$)-full twist with a Legendrian position as the second equation in {\sc Figure}~\ref{tw}.\qed
\end{proof}
\begin{figure}[htbp]
\begin{center}\includegraphics[width=.5\textwidth]{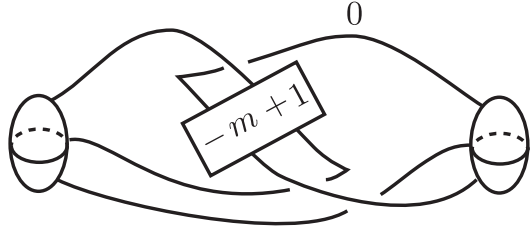}
\caption{Stein structure on $C(m)$.}
\label{Ackstein}
\end{center}
\end{figure}
\begin{figure}[htbp]
\begin{center}
\includegraphics{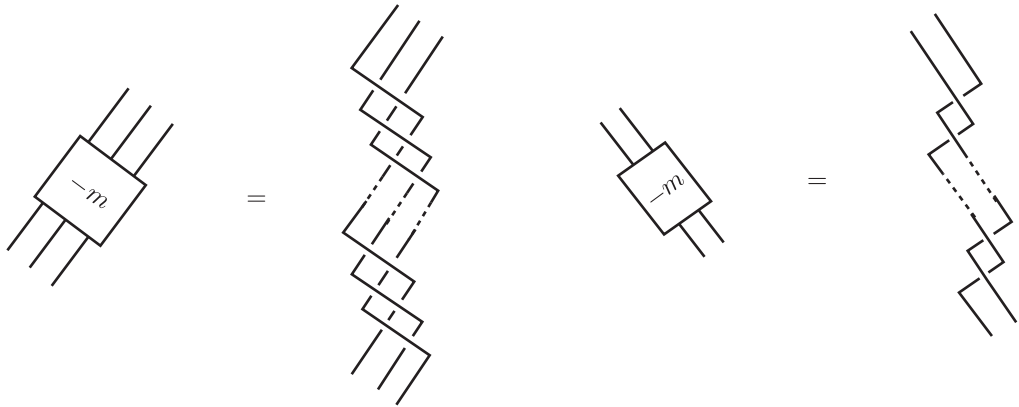}
\caption{Legendrian positions of strings with negative twists.}
\label{tw}
\end{center}
\end{figure}

We prove Theorem~\ref{main}.\\
{\bf Proof of Theorem~\ref{main}.}
We deform the handle diagram in {\sc Figure}~\ref{defofCn} as in {\sc Figure}~\ref{isotopytoStein}.
The last picture in {\sc Figure}~\ref{isotopytoStein} can be easily changed by isotopy into a Legendrian link in the standard position of a connected sum of several copies of $S^2\times S^1$ as in {\sc Figure}~\ref{gestein}.
For the case of $n=4$, the handle diagram is {\sc Figure}~\ref{stein}.

We show that for any integer $i$ with $1\le i\le n-1$, $(\tau_{n,m}^C)^i$ cannot extend to inside $C_{n,m}$.
One attaches a $-1$-framed 2-handle $h$ to the meridian of $\beta_{n-i}$ in $C_{n,m}$.
The resulting 4-manifold 
$$Z_{n-i}=C_{n,m}\cup_{\beta_{n-i}}h$$
is a Stein 4-manifold.
However the attaching sphere of the corresponding $\beta_{n-i}$ in $Z_{n-i}(C_{n,m},(\tau_{n,m}^C)^i)$ is a meridian of the $0$-framed 2-handle.
Thus we can construct an embedded sphere with self-intersection number $-1$
by taking the union of the core disk of $h$ and compressing disk of the meridian in $\partial C_{n,m}$.
On the other hand, in any Stein 4-manifold, there never exist any embedded $-1$-sphere, for example see \cite{AM2}.
Thus $Z_{n-i}(C_{n,m},(\tau_{n,m}^C)^i)$ never admit any Stein structure, hence $Z_{n-i}$ and $Z_{n-i}(C_{n,m},(\tau_{n,m}^C)^i)$ are exotic 4-manifolds.
This implies $(\tau_{n,m}^C)^i$ cannot extend to inside $C_{n,m}$.
Since $(\tau_{n,m}^C)^n=\text{id}$, clearly $(\tau_{n,m}^C)^n$ can extend to a diffeomorphism on $C_{n,m}$.
Therefore $(C_{n,m},\tau_{n,m}^C)$ is an order $n$ cork.
\qed
\begin{figure}[htbp]
\begin{center}\includegraphics[width=.9\textwidth]{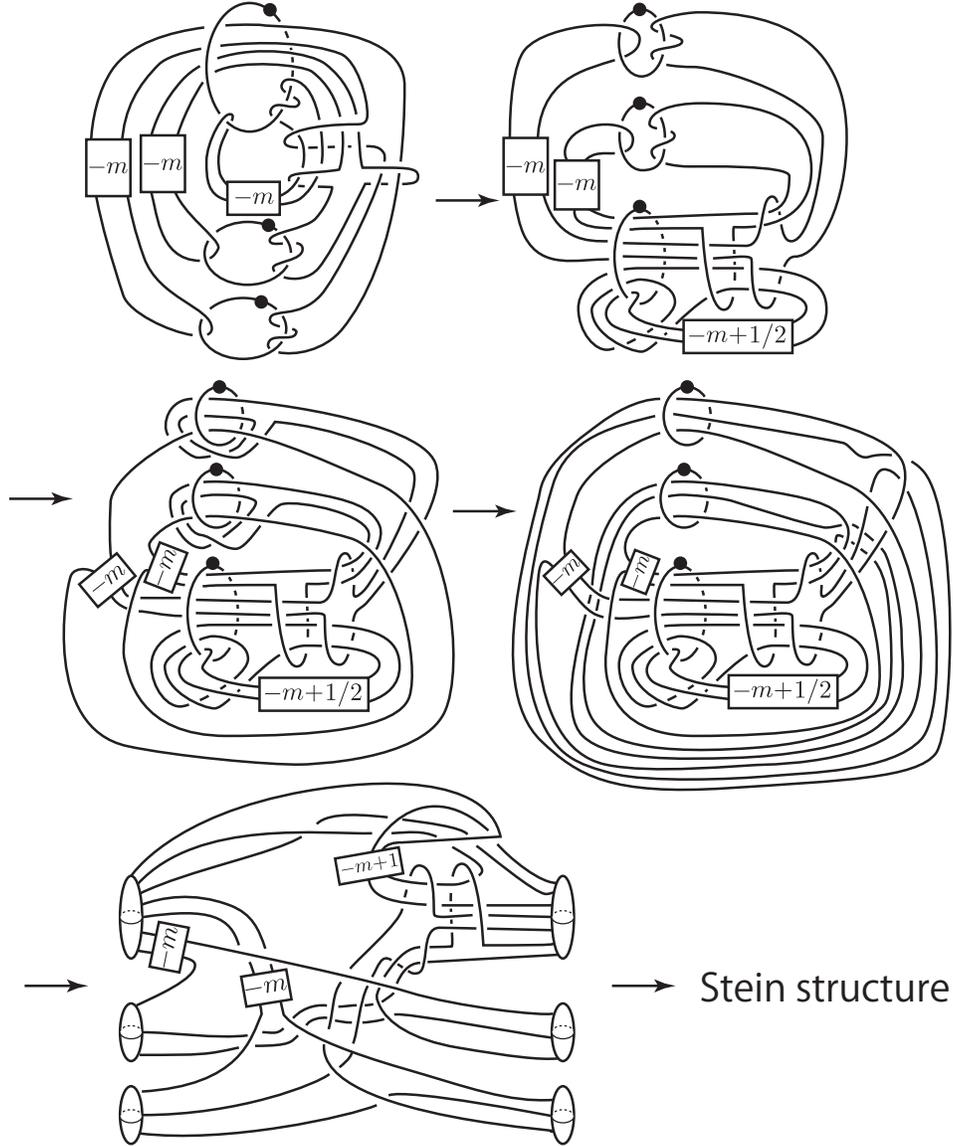}
\caption{The isotopy of $C_{3,1}$ to the Stein structure.}
\label{isotopytoStein}
\end{center}
\end{figure}
\begin{figure}[htbp]
\begin{center}\includegraphics[width=.6\textwidth]{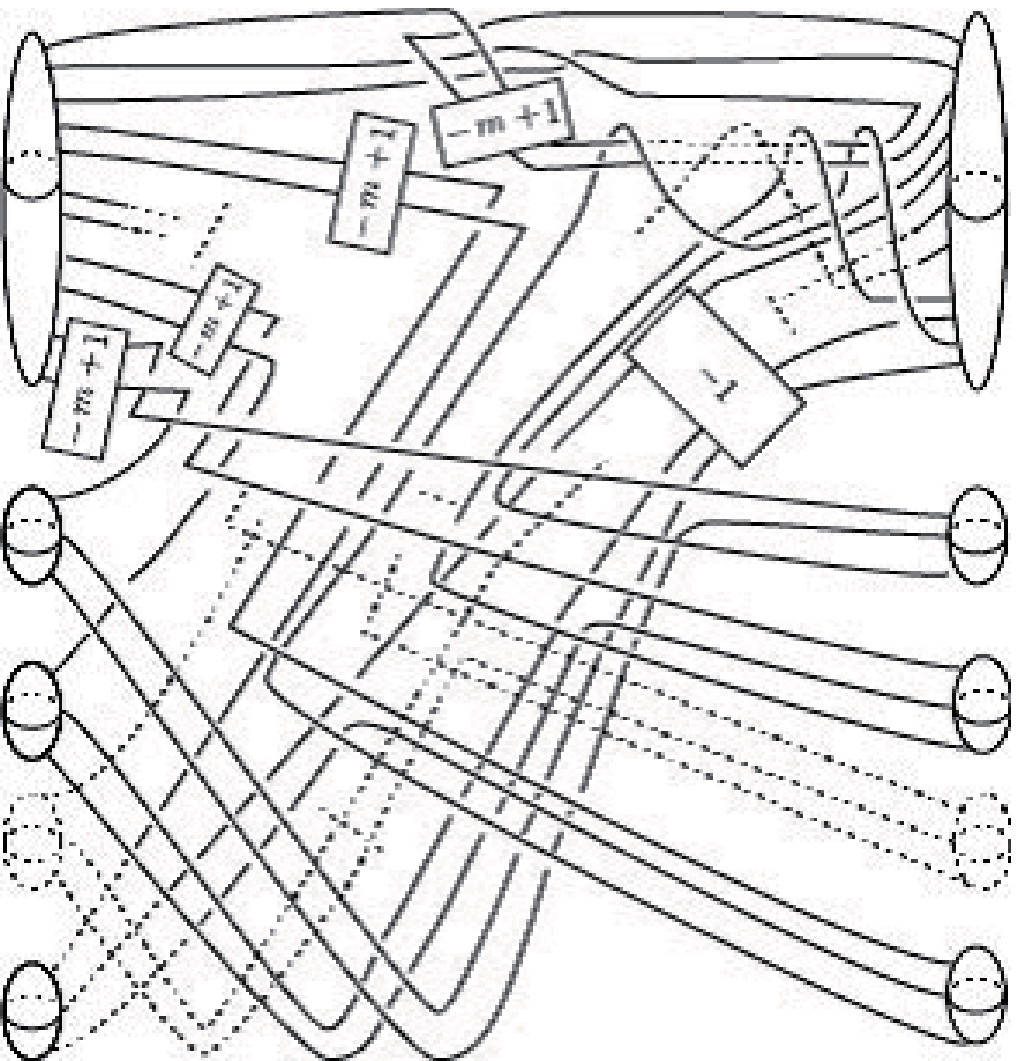}
\caption{A Stein structure of $C_{n,m}$.}
\label{gestein}
\end{center}
\end{figure}
\begin{figure}[htbp]
\begin{center}\includegraphics[width=.6\textwidth]{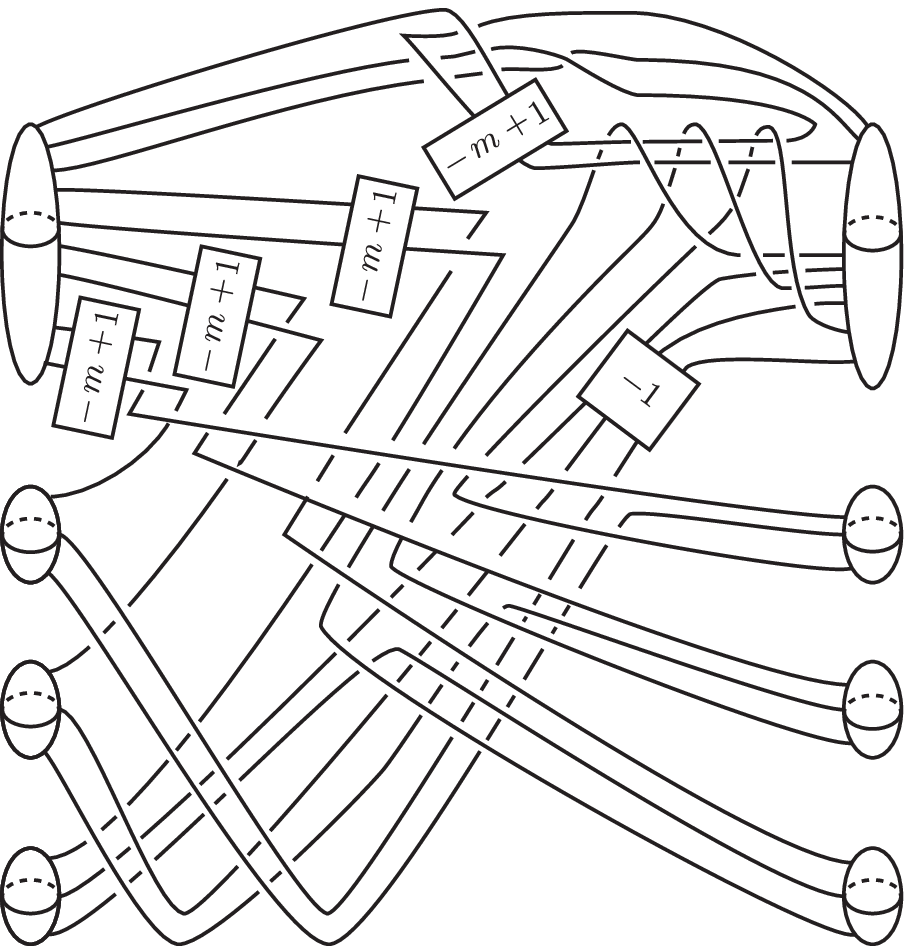}
\caption{A Stein structure of $C_{4,m}$.}
\label{stein}
\end{center}
\end{figure}

The following Lemma~\ref{pcork} and Corollary~\ref{allprecork} are key lemma and corollary for the results on cork in this article.
\begin{lem}
\label{pcork}
Let ${\bf x}$ be a $\{\ast,0\}$-sequence with ${\bf x}\neq (\ast,\ast,\cdots,\ast),(0,0,\cdots,0)$.
Then the twist $(X_{n,m}({\bf x}),\tau_{n,m}^X)$ is a cork twist.
\end{lem}
\begin{proof}
By permuting ${\bf x}$ as $(x_0,x_1,\cdots,x_{n-1})\mapsto (x_1,\cdots,x_{n-1},x_0)$ in several times, we may assume that the sequence is ${\bf x}=(\ast,0,x_2,\cdots,x_{n-1})$.

For $2\le i\le n-1$ if $x_i=\ast$, then attaching a 0-framed 2-handle to the meridian of $\alpha_i$ in $X_{n,m}({\bf x})$ and canceling with $\alpha_i$, we can make a separated 0-framed 2-handle in the diagram of $X_{n,m}({\bf x})\cup\text{2-handle}$.
Canceling the 0-framed 2-handle with a 3-handle, we get $X_{n,m}({\bf x})\cup\text{2-handle}\cup\text{3-handle}=X_{n-1,m}({\bf x}')$, where ${\bf x}'$ is $(\ast,0,x_2,\cdots,\hat{x}_i,\cdots,x_{n-1})$.
The hat means deleting of the component.
This handle attachment means $X_{n,m}({\bf x})\subset X_{n-1,m}({\bf x}')$.
The cobordism between $\partial X_{n,m}({\bf x})$ and $\partial  X_{n-1,m}({\bf x}')$ is a homology cobordism.
Iterating this process, we have $X_{n,m}({\bf x})\subset X_{n',m}(\ast,0,0,\cdots,0)$.

For $2\le i\le n'-1$ if $x_i=0$, then attaching a 2-handle to the meridian of $\alpha_i$ in $X_{n',m}(\ast,0,\cdots,0)$, we can move $\alpha_i$ to the position in the first picture in {\sc Figure}~\ref{0move}.
By doing the handle slide along the arrow in the picture, we get the second picture.
By sliding some components to the meridional 0-framed 2-handle, the components $\{\alpha_i,\beta_i\}$ with the 0-framed 2-handle is separated as in {\sc Figure}~\ref{0move}.
Canceling the 0-framed 2-handle with a 3-handle, we get $X_{n'-1,m}(\ast,0,\cdots,0)$.
Iterating such a canceling, we get $X_{2,m}(\ast,0)$.
Thus, we get an embedding
$$X_{n,m}({\bf x})\subset X_{2,m}(\ast,0)\subset X_{1,m}(\ast)=C(m).$$

Let $\tau_{n,m}^X$ be the $2\pi/n$ rotation as in {\sc Figure}~\ref{beisotopy}.
The twist $C(m)(X_{n,m}({\bf x}),\tau_{n,m}^X)$ is diffeomorphic to $C(m)$, and by the twist, the diffeomorphism on the boundary is mapped in the same way as $\tau(m)$.
If $\tau_{n,m}^X$ extends to a self-diffeomorphism on inside $X_{n,m}$, then $\tau(m)$ can extend to inside $C(m)$.
This is a contradiction to Proposition~\ref{corkorder2}.
Therefore, $(X_{n,m}({\bf x}),\tau_{n,m}^X)$ is a cork.
\begin{figure}[htbp]
\begin{center}
\includegraphics[width=.7\textwidth]{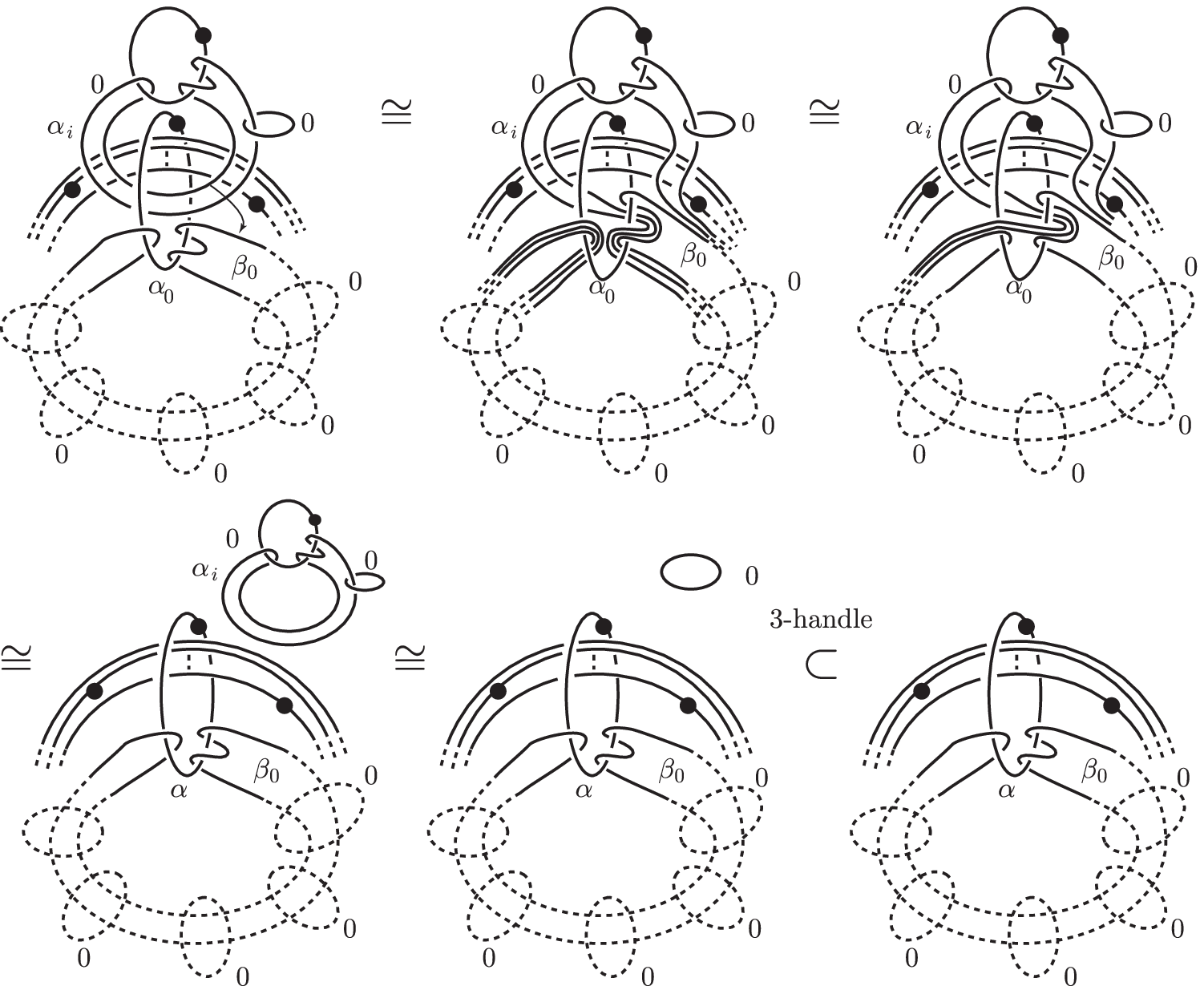}
\caption{$X_{n,m}(\ast,0,\cdots,0)$ with 0-framed 2-handle.}
\label{0move}
\end{center}
\end{figure}
\qed
\end{proof}
Note that these cobordisms give homology cobordisms.
\begin{defn}[Shifting]
Let ${\bf x}$ be a $\{\ast,0\}$-sequence.
If $S_i$ is a cyclic map acting on $\{\ast,0\}$-sequences as ${\bf x}=(x_0,x_1,\cdots,x_{n-1})\mapsto (x_{-i},x_{-i+1},\cdots,x_{-i-1})$, we call $S_i$ a shifting map on the sequence.
Here we consider the suffices as ${\mathbb Z}/n\mathbb{Z}$.
\end{defn}
\begin{defn}[Period]
Let ${\bf x}=(x_0,x_1,\cdots,x_{n-1})$ be a $\{\ast,0\}$-sequence.
We call
$$\min\{p|S_p({\bf x})={\bf x},p>0\}$$
the period of ${\bf x}$.
\end{defn}
From the definition, the period is a divisor of $n$.
\begin{cor}
\label{allprecork}
Let ${\bf x}$ be a $\{\ast,0\}$-sequence with period $N>1$.
Then $(X_{n,m}({\bf x}),\tau_{n,m}^X)$ is an order $N$ cork.
\end{cor}
\begin{proof}
For any integer $0<i<N$, there exists $j$ such that $S_{i}(x_j)\neq x_{j-i}$.
If there does not exist such $j$, then the period of ${\bf x}$ is less than or equal to $i$.
This is contradiction.
Canceling all handles but $\alpha_{j-i}$ and $\beta_{j-i}$ as in the proof in Lemma~\ref{pcork}, we have an embedding:
$$X_{n,m}({\bf x})\subset X_{1,m}(x_{j-i})=C(1).$$
By the shifting map $S_{i}$ the self-diffeomorphism $(\tau_{n,m}^X)^i$ exchanges the dotted 1-handle and 0-framed 2-handle on $\{\alpha_{j-i},\beta_{j-i}\}$, where $\tau_{n,m}^X$ is the $2\pi/n$ rotation, since $S_{i}(x_j)\neq x_{j-i}$.

This twist $C(1)(X_{n,m},(\tau_{n,m}^X)^i)$ for the embedding gives the effect $(C(1),\tau(m))$.
By the same argument as the proof of Lemma~\ref{pcork}, $(X_{n,m}({\bf x}),(\tau_{n,m}^X)^i)$ cannot extend to inside $X_{n,m}({\bf x})$ as a diffeomorphism.

Since shifting map $S_N$ does not change ${\bf x}$, the diffeomorphism $(\tau_{n,m}^X)^N$ can extend to $X_{n,m}({\bf x})$.
As a result, $(X_{n,m}({\bf x}),\tau_{n,m})$ is an order $N$ cork.
\qed
\end{proof}
\subsection{Cork twist for embeddings relationship in $\{C_{n,m}\}$.}
Let $k,n,m$ be positive integers.
There exist embeddings $C_{n,m}\subset C_{n+k,m}$, and $C_{n+k,m}\subset C_{n,m}$ according to Lemma~\ref{pcork}.
\begin{cor}[Cork twist of $C_{n,m}$]
\label{thmC}
Let $n$ be a positive integer.
The cork twist $C_{n,m}(C_{2,m},\tau_{2,m}^C)$ for the first embedding above gives a diffeomorphism $\varphi:C_{n,m}\to C_{n,m}(C_{2,m},\tau_{2,m}^C)$ and the boundary restriction $\varphi|_{\partial}:\partial C_{n,m}\to \partial C_{n,m}$ coincides with $\tau_{n,m}^C$.
\end{cor}
\begin{proof}
This assertion follows immediately from Lemma~\ref{pcork} and the argument below.

Suppose that $n>2$.
The inserting map of $\{\ast,0\}$-sequences
$$(\ast,0)=(x_0,x_1)\mapsto (x_0,y,\cdots,y,x_1),$$
where $y=0$ gives an embedding $C_{2,m}\hookrightarrow C_{n,m}$.
Thus, the cork twist $C_{n,m}(C_{2,m},\tau_{2,m})$ give the effect 
$$(x_0,y,\cdots,y,x_1)\mapsto (x_1,y,\cdots,y,x_0).$$
This deformation corresponds to the cork twist $(C_{n,m},\tau_{n,m}^C)$.

Suppose that $n=1$.
The deleting map of $\{\ast,0\}$-sequences
$$(\ast)\mapsto (0).$$
gives an embedding $C_{2,m}\hookrightarrow C_{1,m}$.
Thus, the cork twist $C_{1,m}(C_{2,m},\tau_{2,m})$ give the effect 
 $(\ast,0)\mapsto (\ast)$
This deformation corresponds to the cork twist $(C(m),\tau(m))$.

Suppose that $n=2$.
By taking the identity map $C_{2,m}\to C_{2,m}$, the statement is trivial.
\begin{figure}[htbp]
\begin{center}\includegraphics[width=.5\textwidth]{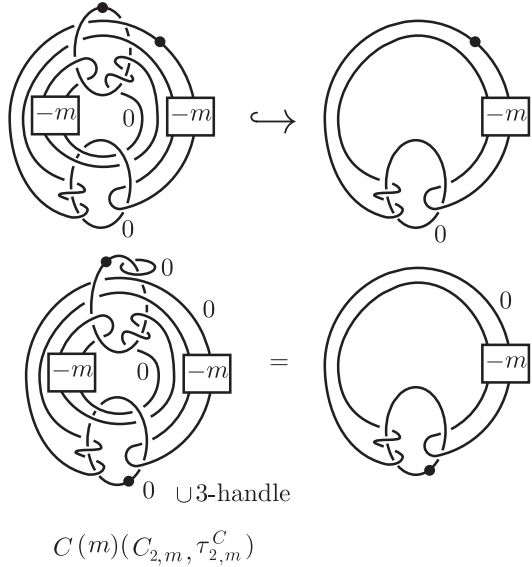}
\caption{An embedding $C_{2,m}\hookrightarrow C(m)$.}
\label{C2m}
\end{center}
\end{figure}
\qed
\end{proof}
\begin{rmk}
\label{Esame}
For $E_{n,m}$, one can find the similar embeddings to those of $C_{n,m}$.
Namely, there exist $E_{n,m}\subset E_{n+k,m}$ and $E_{n,m}\subset E_{n+k,m}$ such that these embeddings satisfy the same condition as that of Proposition~\ref{thmC}
\end{rmk}

{\bf Proof of Theorem~\ref{contra}.}
By using the deleting map $(0,\ast,\cdots,\ast)\mapsto (0,\ast)$ in Lemma~\ref{pcork}, the current theorem is satisfied.
Therefore, as a corollary of Lemma~\ref{pcork} and Corollary~\ref{allprecork} we deduce $(D_{n,m},\tau_{n,m})$ is an order $n$ cork.
\qed

{\bf Proof of Theorem~\ref{thmE}.}
Attaching 2-handles as the first diagram in {\sc Figure}~\ref{handlenejire3}, we obtain the last diagram in the figure.
The 4-manifolds for this diagram can be embedded in $E(l)\#n\overline{{\mathbb C}P^2}$.
This can be seen due to Figure~9.4 in \cite{GS}.
Actually, if $l\ge \lceil\frac{2n+1}{3}\rceil$, then this embedding into $E(l)\# n\overline{{\mathbb C}P^2}$ can be constructed.

We show that the twist by an embedding $E_{n,m}\hookrightarrow V_{l,n}$ produces $(2l-1){\mathbb C}P^2\#(10l+n-1)\overline{{\mathbb C}P^2}$ by using the argument in Exercise~9.3.4 in \cite{GS}.
$E_{n,m}\hookrightarrow V_{l,n}(E_{n,m},(\tau_{n,m}^E)^i)$ contains the handle diagram as in {\sc Figure}~\ref{CP2appear}.
By doing the handle slide as indicated by the arrow in the right hand side in {\sc Figure}~\ref{CP2appear}, we get a $\mathbb{C}P^2$ connected-sum component.
By using a $\mathbb{C}P^2$ formula ({\sc Figure}~\ref{aformula}), we get the third diagram in {\sc Figure}~\ref{CP2appear}.
A separated $\mathbb{C}P^2$ component in the third diagram can be moved to the position before separating by using the converse of the handle moves from the second picture to the third picture in {\sc Figure}~\ref{CP2appear}.
For the last figure in {\sc Figure}~\ref{CP2appear} we can use the method as Exercise 9.3.4. in \cite{GS}.
As a result, for $0<i\le n-1$ we have a diffeomorphism $V_{l,n}(E_{n,m},(\tau_{n,m}^E)^i)\cong (2l-1){\mathbb C}P\#(10l+n-1)\overline{\mathbb{C}P^2}\not\cong V_{n,l}$.
This means that $(E_{n,m},\tau_{n,m}^E)$ is an order $n$ cork.
\qed
\begin{figure}[htbp]
\begin{center}\includegraphics[width=.85\textwidth]{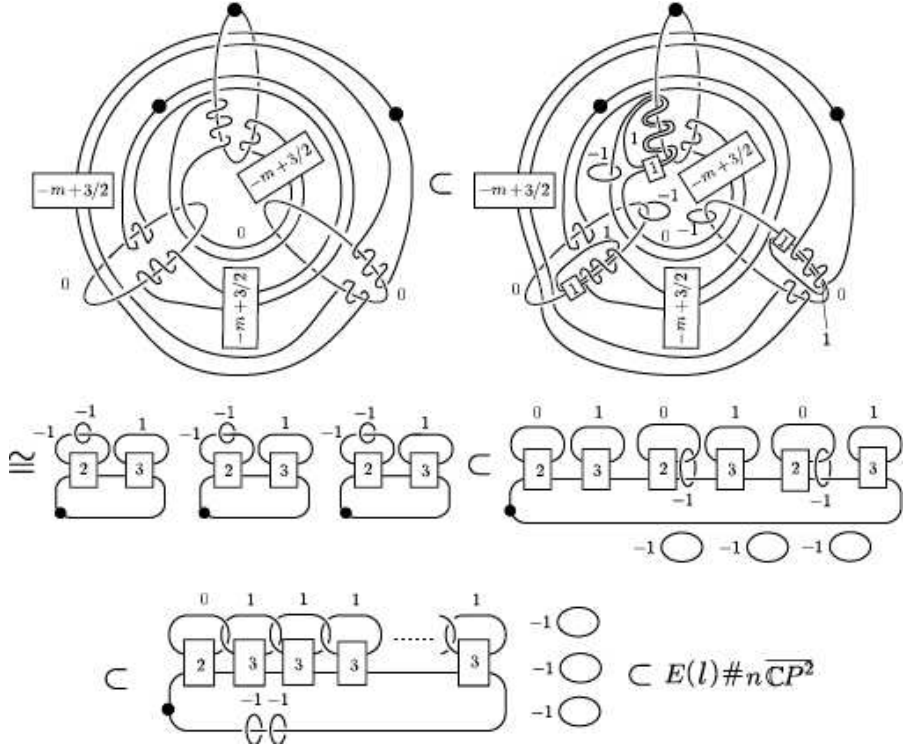}
\caption{An embedding $E_{n,m}\hookrightarrow E(l)\#n\overline{{\mathbb C}P^2}$.}
\label{handlenejire3}
\end{center}
\end{figure}
\begin{figure}[htbp]
\begin{center}\includegraphics[width=.85\textwidth]{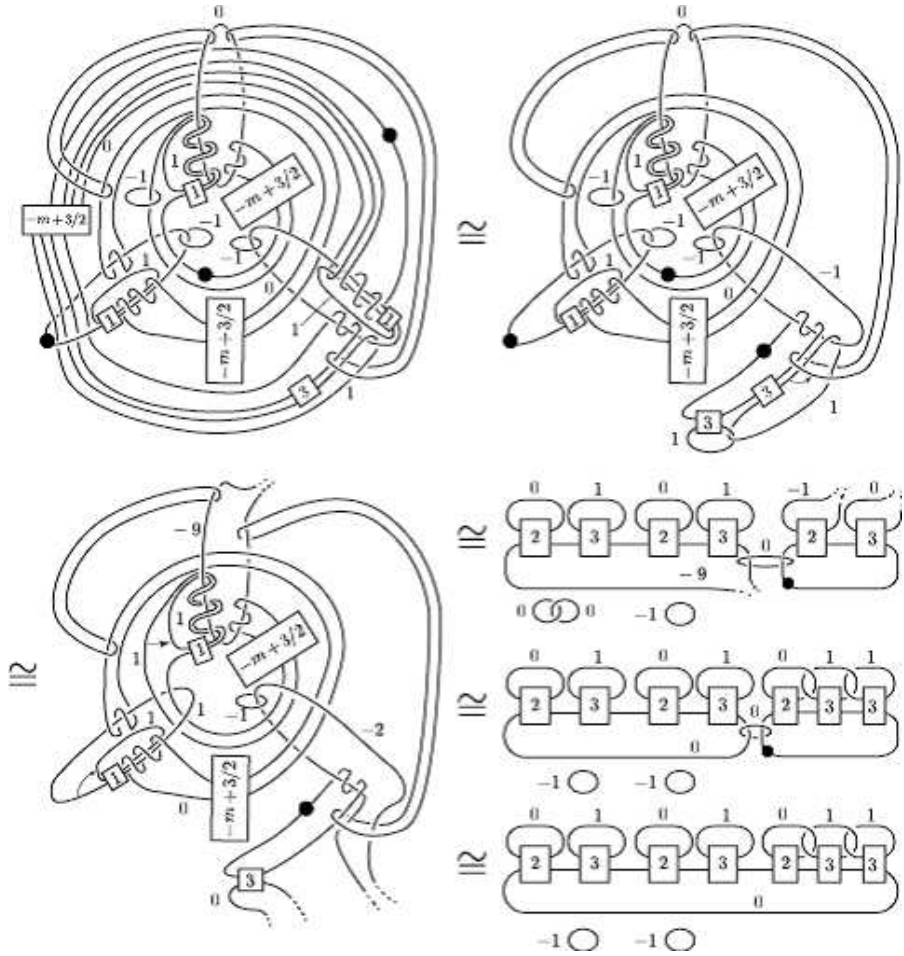}
\caption{A submanifold $E(l)\#n\overline{{\mathbb C}P^2}(E_{n,m},(\tau_{n,m}^E)^i)$.}
\label{CP2appear}
\end{center}
\end{figure}
\begin{figure}[htbp]
\begin{center}\includegraphics[width=.5\textwidth]{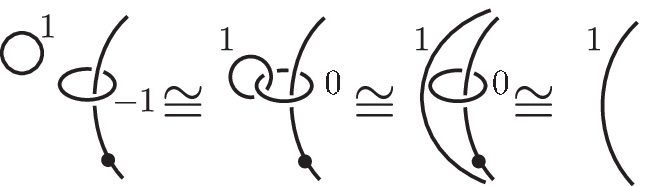}
\caption{A separated $\mathbb{C}P^2$ formula.}
\label{aformula}
\end{center}
\end{figure}

\begin{rmk}
Proposition~\ref{thmC} and Theorem~\ref{contra} imply that even when $C(m)$ cannot be embedded in a 4-manifold $X$,
if the diagram of $C(m)$ is contained in $X$ as a sub-diagram and $C_{n,m}\hookrightarrow X$ or $D_{n,m}\hookrightarrow X$ with respect to the sub-diagram, then by doing cork twist $(C_{n,m},\tau_{n,m}^C)$ or $(D_{n,m},\tau_{n,m})$ respectively, exchanging the dot and 0 in the sub-diagram of $X$ can be realized by an cork twist $C_{n,m}$ and $D_{n,m}$.
\end{rmk}
\subsection{Exotic 4-manifolds $W_{n,m,i}$.}
\label{exo}
If there does not exist any smoothly embedded $-1$-sphere in a 4-manifold, then we say the 4-manifold is minimal.

{\bf Proof of Proposition~\ref{collexo}.}
Let $W_{n,m}$ be a 4-manifold $C_{n,m}$ with $n(n-1)/2$ 2-handles attached.
For $1\le i\le n-1$, the attaching spheres are the $i$ parallel meridians of $\beta_i$ in the diagram of $C_{n,m}$ in {\sc Figure}~\ref{beisotopy}.
The framings are all $-1$.
The diagram of $W_{3,m}$ is {\sc Figure}~\ref{handle2}.
The manifold $W_{n,m}$ is simply-connected and $b_2=n(n-1)/2$ and intersection form is the $b_2$ direct sum of $\langle -1\rangle$. 
Clearly $W_{3,m}$ is a Stein manifold and in particular minimal.
Performing the cork twist $(C_{n,m},(\tau_{n,m}^C)^i)$ for $W_{n,m}$, we get a 4-manifold 
$$W_{n,m,i}=W_{n,m}(C_{n,m},(\tau_{n,m}^C)^i).$$
By this cork twist, the $i$ parallel meridians are moved to the parallel meridians of $\beta_0$, which is a $0$-framed 2-handle.
Thus the $i$ meridians can be blow-downed.
Hence, we get $W_{n,m,i}=W_{n,m,i}'\#\,i\,\overline{\mathbb{C}P^2}$.
Thus $W_{n,m,i}$ and $W_{n,m}$ are exotic, because $W_{n,m}$ is minimal.
\qed

The problem of whether $\{W_{n,m,i}\}_{i=0,\cdots,n-1}$ are mutually exotic 4-manifolds is remaining.
If all $W_{n,m,i}'$ for any $i$ are minimal, then $W_{n,m,i}$ are mutually exotic 4-manifolds, i.e., $(C_{n,m},\tau_{n,m}^C)$ is a cork for the collection $\{W_{n,m,i}\}$.
\begin{figure}[htbp]
\begin{center}
\includegraphics[width=.3\textwidth]{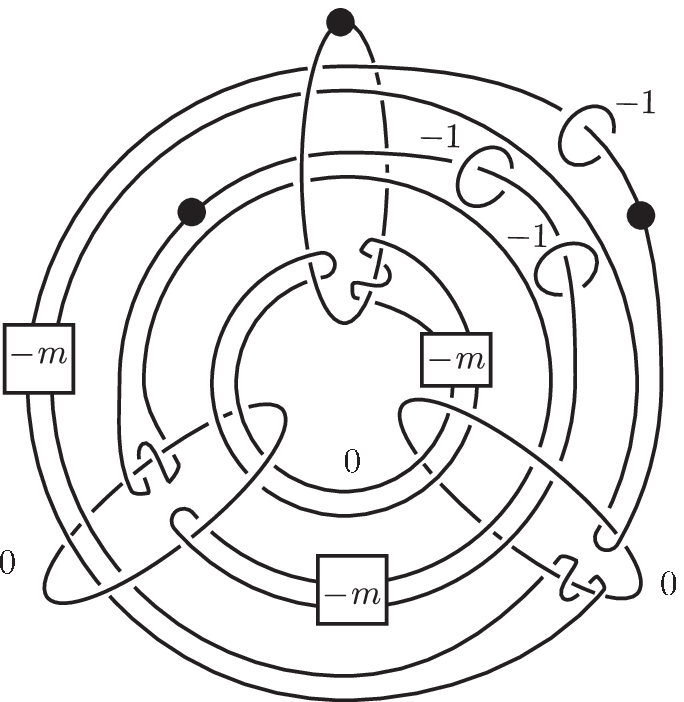}
\caption{The handle decomposition of $W_{3,m}$.}
\label{handle2}
\end{center}
\end{figure}

{\bf Proof of Proposition~\ref{exm1}.}
Let $F=F_{2,m}$ be a contractible 4-manifold defined in Definition~\ref{CDF}.
Let $\kappa=\tau_{2,m}^F$ be a rotation by $\pi/2$.
Since $(\ast,0,\ast,0)$ is a period 2 $\{\ast,0\}$-sequence, $(F,\kappa)$ is order $2$ cork by Lemma~\ref{allprecork}.
\qed
\begin{figure}[htbp]
\begin{center}\includegraphics[width=.4\textwidth]{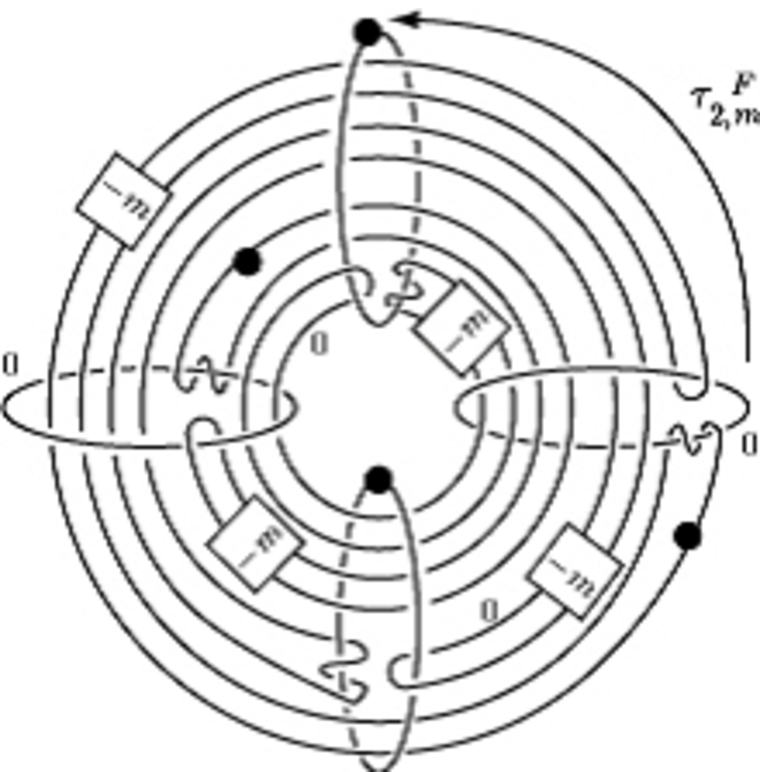}
\caption{A contractible 4-manifold $F_{2,m}$ and a diffeomorphism $\tau_{2,m}^F$ with order 4 as a diffeomorphism.}
\label{orderdiff}
\end{center}
\end{figure}
\section{A non-trivial action on $HF^+(\partial C_{n,m})$.}
In this section, as an application of finite order corks we prove Theorem~\ref{action}.
This theorem is a generalization of the main theorem in \cite{AK}.
The terms used here are the same ones as those in \cite{AK}.
The argument is parallel to Theorem~4.1 in \cite{AK}.

{\bf Proof of Theorem~\ref{action}.}
We may show that the map $(\tau_{n,m}^C)^i$ induces a non-trivial action on $HF^+(\partial C_{n,m})$.
Let $\tau_i$ denote $(\tau_{n,m}^C)^i$.
Let $\xi$ be the contact structure on $\partial C_{n,m}$ induced from the Stein structure on $C_{n,m}$.
As described as in {\sc Figure}~2 in \cite{AK}, we attach a 2-handle $H$ along a trefoil linked with $\beta_{n-i}$ with framing 1.
We denote $U_{i}=C_{n,m}\cup_{\beta_{n-i}} H$.
This manifold is a Stein manifold because the maximum Thurston-Bennequin number of the trefoil is 1 and the Stein structure of $C_{n,m}$ extends to $H$.
\begin{figure}[htbp]
\begin{center}
\includegraphics[width=.7\textwidth]{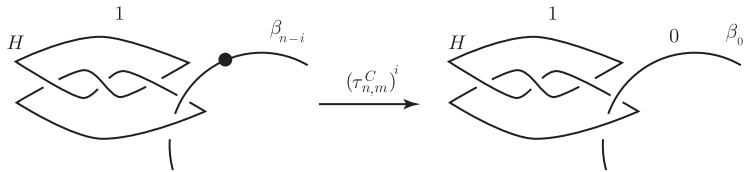}
\caption{The twist of $U_i$ via $(\tau_{n,m}^C)^i$.}
\label{handle2}
\end{center}
\end{figure}

Let $V$ be a concave extension of $(\partial C_i,\xi)$ of $H$ due to \cite{LM}.
Thus, $X_i=C_{n,m}\cup V$ is a closed symplectic structure with $b_2^+>1$.
We define the twist via $\tau_i$ by $X'_i=C_{n,m}\cup_{\tau_i} V$.
In this manifold, we can find an embedded self-intersection number 1 torus.
Here we use the following theorem.
Here $\Phi_{X,\frak{s}}$ is the Ozsv\'ath-Szab\'o 4-manifold invariant.
\begin{thm}[Ozsv\'ath-Szab\'o \cite{OS}]
Let $X$ be a closed 4-manifold.
Let $\Sigma\subset X$ be a homologically non-trivial embedded surface with genus $g\ge 1$ and with non-negative self-intersection number.
Then for each spin$^c$ structure $\frak{s}\in\text{Spin}^c(X)$ for which $\Phi_{X,\frak{s}}\neq0$, we have that
$$ |\langle c_1(\frak{s}), [\Sigma]\rangle | + [\Sigma] \cdot [\Sigma] \le 2g-2$$
The following is another version of the adjunction inequality along with a non-vanishing
result of the 4-manifold invariant for Lefschetz fibrations.
\end{thm}
Hence, the Ozsv\'ath-Szab\'o 4-manifold invariant for $X'$ has no basic class.
Here we obtain the same computation as in \cite{AK}:
$$F_{C_{n,m},\frak{s}_0}^+(c^+(\xi))=\pm F^+_{C_{n,m},\frak{s}_0}\circ F^{\text{mix}}_{V,\frak{s}}(\Theta_{(-2)}^-)=F^{\text{mix}}_{X,\frak{s}}(\Theta_{(-2)}^-)=\pm\Theta_{(0)}^+,$$
and on the other hand,
$$ F_{C_{n,m},\frak{s}_0}^+(\tau_i^\ast(c^+(\xi)))=F^+_{C_{n,m},\frak{s}_0}\circ \tau_i^\ast\circ F^{\text{mix}}_{V,\frak{s}}(\Theta_{(-2)}^-)=\pm F^{\text{mix}}_{X',\frak{s}'}(\Theta_{(-2)}^-)=0,$$
where $\frak{s}$ is the canonical spin$^c$ structure on symplectic manifold $X$ and $\frak{s}'$ is structure, the induced spin$^c$ structure on $X'$,
and $\frak{s}_0$ is the restriction of $\frak{s}$ and $\frak{s}'$ to $C_{n,m}$.
Here $c^+(\xi)$ is Ozsv\'ath-Szab\'o's contact invariant for $(\partial C_{n,m}^C,\xi)$.
We use the equality $F^{\text{mix}}_{W,\xi}(\Theta_{(-2)}^-)=\pm c^+(\xi)$ in \cite{P}, where $\xi$ is a contact structure on $\partial W$ with torsion $c_1$.

These two inequalities above means that $c^+(\xi)$ and $\tau_i^\ast (c^+(\xi))$ are distinct elements.
Thus, $\tau_{n,m}^C$ act on $HF^+(\partial C_{n,m})$ effectively.
Namely, $\xi,\tau_1^\ast(\xi),\cdots,\tau_{n-1}^\ast(\xi)$ are distinct elements.
\qed

{\bf Proof of Proposition~\ref{contact}.}
Let $\xi$ be a contact structure on $\partial C_{n,m}$ induced by the Stein structure.
Let $\xi_i$ denote $\tau_i^\ast(\xi)$.
Each two structures $\xi_i,\xi_j$ are homotopic 2-plane fields.
For, there exist a trivial cobordism between these contact 3-manifolds, hence $\xi_i,\xi_j$ are the same 3-dimensional invariants by \cite{G}.
This means that these contact structures are homotopic each other.
The diffeomorphism $\tau_{i-j}$ gives a contactomorphism from $\xi_i$ to $\xi_j$.
However, $\xi_i,\xi_j$ are not isotopic because $c^+(\xi_i)$ and $c^+(\xi_j)$ are distinct.
Furthermore, $\xi_i$ is Stein filling.
\qed
\section{Problems}
Here we raise the two problems.
\begin{prob}
Show that $(C_{n,m},\tau_{n,m}^C)$ is a cork for the collection $\{W_{n,m,i}\}_{i=0,1,\cdots,n-1}$.
\end{prob}
\begin{prob}
Show that $D_{n,m},E_{n,m}$ are also finite order Stein corks.
In general for any $\{\ast,0\}$-sequence ${\bf x}$, show that $X_{n,m}({\bf x})$ is a Stein manifold.
\end{prob}

\end{document}